\documentclass[12pt]{amsart}
\copyrightinfo{2010}{Luke Oeding}
\usepackage{graphicx}
\usepackage{url,amstext,amsfonts,amssymb,amscd,amsbsy,amsmath,amsthm,verbatim,mathrsfs}
\usepackage{ifthen, fullpage, mathrsfs, cite,hyperref}
\usepackage{color}
\usepackage[enableskew]{youngtab} 

\newtheorem{theorem}{Theorem}[section]
\newtheorem{conj}[theorem]{Conjecture}

\newtheorem{prop}[theorem]{Proposition}

\newtheorem{computation}[theorem]{Computation}

\theoremstyle{definition}

\newtheorem{example}[theorem]{Example}

\theoremstyle{remark}
\newtheorem{remark}[theorem]{Remark}
\newcommand{\CC}{\mathbb C}

\newcommand{\PP}{\mathbb P}
\newcommand{\V}{\mathcal{V}}
\newcommand{\I}{\mathcal{I}}

\newcommand{\Sub}{\operatorname{Sub}}

\begin{document}

\title{Toward a salmon conjecture}
\author{Daniel J. Bates$^{1}$}
\address{Department of Mathematics, Colorado State University, Fort Collins, CO 80523 (USA)}
\email{bates@math.colostate.edu}
\thanks{$^{1}$partially supported by NSF grant DMS--0914674.}

\author{Luke Oeding$^{2}$}
\address{
Dipartimento di Matematica ``U. Dini'', Universit\`a di Firenze, Viale Morgagni 67/A, 50134 Firenze (Italy)}
\email{oeding@math.unifi.it}
\thanks{$^{2}$supported by National Science Foundation grant Award No. 0853000: International Research Fellowship Program (IRFP)}

\begin{abstract}
By using a result from the numerical algebraic geometry package Bertini we show that (up to high numerical accuracy) a specific set of degree $6$ and degree $9$ polynomials cut out the secant variety $\sigma_{4}(\PP^{2}\times \PP ^{2} \times \PP ^{3})$. This, combined with an argument provided by Landsberg and Manivel (whose proof was corrected by Friedland), implies set-theoretic defining equations in degrees $5$, $6$ and $9$ for a much larger set of secant varieties, including $\sigma_{4}(\PP^{3}\times \PP ^{3} \times \PP ^{3})$ which is of particular interest in light of the salmon prize offered by E. Allman for the ideal-theoretic defining equations.\end{abstract}

\maketitle

\section{Introduction}

In 2007, E. Allman offered a prize of Alaskan salmon to anyone who finds the defining ideal of the following secant variety \[
\sigma_{4}\left(\PP^{3}\times \PP ^{3}\times \PP^{3}\right)
,\]
\cite{prize}.  Recall that if $A,B,C$ are vector spaces, then the Segre product is defined by the following embedding into the tensor product
\begin{align*}
Seg\colon \PP A \times \PP B \times \PP C &\rightarrow \PP (A\otimes B\otimes C) \\
([a],[b],[c]) &\mapsto [a\otimes b\otimes c]
.\end{align*}
Further recall that if $X \subset \PP^{N}$ is a variety, then the $k$-secant variety of $X$ denoted  $\sigma_{k}(X)\subset\PP^{N}$, is the Zariski closure of all points on secant $\PP^{k-1}$'s to $X$.
For simplicity, we will drop the reference to the Segre embedding and write $\sigma_{k}(\PP A\times \PP B \times \PP C)$ for the secant variety to the Segre product.
%
  Secant varieties have been studied classically, but we have a renewed interest in their study because of the salmon prize and other related recent works on the subject (see \cite{LM08,CGG5, LBull, AllmanRhodes08, Friedland2010, LWSecSeg, SS}).

Allman's ideal-theoretic question is still open. Our main result is Theorem \ref{thm:main}, in which we give a geometric argument (relying on results of Landsberg and Manivel \cite[Corollary 5.6]{LM08} and the recent correction of the proof by Friedland in \cite{Friedland2010}) combined with a calculation using numerical algebraic geometry to show that up to high numerical accuracy, $\sigma_{4}\left(\PP^{3}\times \PP ^{3}\times \PP^{3}\right)$  is cut out set-theoretically by $1728$ equations in degree $5$, $1000$ equations in degree $6$ and $8000$ equations in degree $9$.  
Even though these dimensions are large, we show that in each degree, the large space of polynomials can be constructed from a small number of representatives via substitutions (see Remarks \ref{M6 basis}, \ref{M5 basis} and \ref{M9 basis}).
Theorem \ref{thm:main} solves the set-theoretic version of Allman's question (up to high numerical accuracy), uses equations of lower degree than Friedland's solution \cite{Friedland2010}, and gives evidence for a conjecture to the ideal-theoretical question asked by Allman.

One practical interest of the secant variety $\sigma_{4}\left(\PP^{3}\times \PP ^{3}\times \PP^{3}\right)$ is in phylogenetics, where the secant variety is associated to the statistical model for evolution called the mixture model of independence models \cite{AllmanRhodes08}. The main motivation to study this particular model is that Allman and Rhodes showed in  \cite[Theorem~11]{AllmanRhodes08} that finding the polynomial invariants for this small evolutionary tree would provide all polynomial invariants for the statistical model for \emph{any} binary evolutionary tree with \emph{any} number of states.  

Note that in this paper we work exclusively over the complex numbers, however, in phylogenetics one is often interested in studying models restricted to the real numbers, the positive real numbers, or the probability simplex.  Since equations for a given model considered over the complex numbers also provide equations for the restricted model, it is natural to start with the complex setting and then study the additional necessary equations and inequalities imposed by the given restriction.  We leave this further study to other works.

While Allman asks for the generators of the defining ideal of the secant variety, a collection of set-theoretic defining equations provides a necessary and sufficient test for membership on the model.
Very recently Friedland \cite{Friedland2010} has proved (without a computer) that a set of polynomials in degrees $5, 9$ and $16$ define $\sigma_{4}\left(\PP^{3}\times \PP ^{3}\times \PP^{3}\right)$ set-theoretically. Indeed, Friedland's set of polynomials do (in theory) allow one to test whether a given set of data fits the model.  
Because it uses polynomials in smaller degree, Theorem~\ref{thm:main} provides a more efficient practical membership test for the model. 

On the other hand, Casanellas and Fernandez-Sanchez \cite{Casanellas} have studied
more practical issues regarding phylogenetic tree construction using algebraic methods.  In particular, they point out that for phylogenetic tree reconstruction, the equations coming from the edges of the tree (minors of flattenings below) seem to be more relevant than the equations coming from vertices (the equations of degrees 5 and 9 are examples of such).

Our equations in degree $6$ are not in the ideal of the equations in degree $5$, thus they are non-trivial generators in the ideal, and Friedland's result cannot be a set of minimal generators of the ideal.  We have not found any such obstructions to our result holding ideal-theoretically and this leads to a salmon conjecture that the ideal-theoretic version of Theorem~\ref{thm:main} also holds.

This work was initiated in October 2008 when Bernd Sturmfels asked for a Macaulay2 readable file of the degree $6$ polynomials in the ideal of $\sigma_{4}(\PP^{2}\times \PP^{2}\times \PP^{3})$.  Proposition \ref{prop:deg6ideal} is a representation theoretic description of these polynomials and corrects minor errors in \cite[Proposition~6.3]{LM04}, and \cite[Remark~5.7]{LM08}. In Section~\ref{sec:polys} we give a brief overview of how these polynomials were constructed from their representation theoretic description. 
These equations and other ancillary materials for this paper are available in the ancillary materials which accompany the arXiv version of this paper or by contacting either author.

At the December 2008 MSRI workshop on Algebraic Statistics, Oeding presented Conjecture \ref{conj:m6} which, when combined with an argument of Landsberg and Manivel, implies our main result.  This argument is discussed in Section~\ref{sec:geometry}.  The missing ingredient for the conjecture was to understand the zero-set of the degree $6$ polynomials. Shortly after this workshop, the Oeding asked for help from Bates and the Bertini Team.  

The two authors worked together to get the correct mixture of initial input and computing strategies in order to find a computation that would finish in a reasonable amount of time.  Finally on July 12, 2010, a computation that had taken approximately 2 weeks on 8 processors (two 2.66 GHz quad-core Xeon 5410s set up as one head processor and seven worker processors) finished, providing a numerical proof to Conjecture \ref{conj:m6}.  Because our calculations use numerical approximations, we say that the proof holds up to high numerical accuracy. In Section~\ref{sec:Bertini} we discuss our computational methods and the reliability of this result.

\section{Symmetry and the equations in degree $6$}\label{sec:polys}
In this section we recall well-known facts about the variety and equations we are studying.  The main purpose is to set up notation.  The reader who is unfamiliar with these concepts may consult \cite{FH}, or for a more detailed account related to secant varieties see \cite{LM04,LM08,LWSecSeg} or the upcoming \cite{LandsbergTensorBook}. 

Let $A,B,C$ be vector spaces of dimensions $a,b,c$ respectively.
The symmetry group of $
\sigma_{r}\left(\PP A \times \PP B \times \PP C \right)$ 
is the change of coordinates in each factor
$ GL(A)\times GL(B) \times GL(C)$
(or when $A\cong B \cong C$ there is an additional symmetric group $\mathfrak{S}_{3}$ acting and the symmetry group is  $\left(GL(A)\times GL(B) \times GL(C)\right) \ltimes \mathfrak{S}_{3}$  ).
Therefore we can use tools from representation theory to aid in our search for defining equations.  Since much of this work has already been done, we only describe the equations relevant for our application.

The module $S^{d}(A^{*}\otimes B^{*} \otimes C^{*})$ of degree $d$ homogeneous polynomials on $A\otimes B\otimes C$ has an \emph{isotypic decomposition} (see \cite[Proposition~4.1]{LM04})
\[
S^{d}(A^{*}\otimes B^{*} \otimes C^{*} ) = \bigoplus_{|\pi_{1}|=|\pi_{2}|=|\pi_{3}|=d}  \left(S_{\pi_{1}}A^{*} \otimes S_{\pi_{2}}B^{*} \otimes S_{\pi_{3}} C^{*} \right) ^{\oplus m_{\pi_{1},\pi_{2},\pi_{3}}}
,\]
where the $\pi_{i}$ are partitions of $d$ and the multiplicity $m_{\pi_{1},\pi_{2},\pi_{3}}$ is the dimension of the highest weight space which can be computed via characters.
The modules
\[
(S_{\pi_{1}}A^{*} \otimes S_{\pi_{2}}B^{*} \otimes S_{\pi_{3}} C^{*}) ^{m_{\pi_{1},\pi_{2},\pi_{2}}}
\]
are called \emph{isotypic components}, and the individual modules $S_{\pi_{1}}A^{*} \otimes S_{\pi_{2}}B^{*} \otimes S_{\pi_{3}} C^{*}$ are irreducible $GL(A)\times GL(B) \times GL(C)$-modules sometimes called \emph{Schur modules}. 

The ideal of any $ GL(A)\times GL(B) \times GL(C)$-invariant variety in $\PP(A\otimes B\otimes C)$ consists of a subset of the modules occurring in the isotypic decomposition.
If $X$ is a projective variety, let  $\I_{s}(X)$ denote the ideal of homogeneous degree $s$ polynomials in the ideal of $X$.
In general, if $X$ is any variety with ideal generated in degree $2$ (of which the Segre variety is an example), $\I_{s}(\sigma_{k}(X)) = 0$ for $s\leq k$ (see \cite[Corollary 3.2]{LM04}), and in particular,  $\I_{s}(\sigma_{4}(\PP A \times \PP B \times \PP C)) = 0$ for $s \leq 4$.
Also, one can calculate (by checking every irreducible module of degree $5$ polynomials) that 
\[\I_{5}\left(\sigma_{4}\left( \PP^{2} \times \PP^{2} \times \PP^{3} \right)\right) = 0.\]
In addition, we have found the following.
\begin{prop}\label{prop:deg6ideal}  Let $A \cong B\cong \CC^{3} $, $C\cong \CC^{4}$, and let $M_{6}$ denote the module 
$S_{2,2,2}A^{*} \otimes S_{2,2,2} B^{*} \otimes S_{3,1,1,1} C^{*}$.
Then $M_{6}=\I_{6}\left(\sigma_{4}\left( \PP A \times \PP B \times \PP C \right)\right)$ as $GL(A)\times GL(B) \times GL(C)$-modules. \end{prop}
\begin{proof}
The module $M_{6}$ was found by following the ideal membership test described in \cite{LM04}.  We repeated the procedure outlined in \cite{LM04} as follows.
We first decomposed $S^{6}(A^{*}\otimes B^{*}\otimes C^{*})$ into its isotypic decomposition.  This plethysm calculation can be done in one line on the program LiE \cite{LiE} as 
\begin{verbatim}
 plethysm([6],[1,0,1,0,1,0,0],A2A2A3)
\end{verbatim}
or by using the procedure ``mults'' which we implemented in maple and can be found in the file ``iso\_mults.mw,'' available with our ancillary materials.
Next we computed a basis of the highest weight space for each isotypic component.  We implemented in Maple a standard algorithm to compute a basis of the highest weight space in the image of the relevant Schur functors associated to each module. This implementation is in the file called ``poly\_make\_algo.mw'' which also may be found with our ancillary materials.  A detailed exposition of this concept may be found below. 
 We applied this algorithm to each module in the isotypic decomposition and
 we checked to see if any linear subspace of the highest weight space of an isotypic component vanished on the variety by direct evaluation.  The only module which passed this test was $M_{6}$, which occurs with multiplicity one in $S^{6}(A^{*}\otimes B^{*}\otimes C^{*})$.
\end{proof}

We note that there was some confusion between the statements and proofs in the pre-print and the print version of \cite[Proposition~6.3]{LM04} as well as in the statement \cite[Remark~5.7]{LM08}, and we believe that Proposition \ref{prop:deg6ideal} corrects this confusion.

The module $S_{2,2,2}\CC^{3}$ is one-dimensional and 
as a vector space, the module $S_{3,1,1,1}\CC^{4}$ is isomorphic to $S^{2}\CC^{4}$, which is $10$-dimensional.  Our construction produces a basis of the module $M_{6}$ consisting of $10$ polynomials which also correspond to the $10$ semi-standard fillings (strictly increasing in the columns and non-decreasing in the rows) of the tableau of shape $(3,1,1,1)$ with the numbers $1,2,3,4$. We list these fillings below.  
The basis of polynomials is contained in the file ``deg\_6\_salmon.txt'' which is available with our ancillary materials as mentioned above.

Here is a brief overview of an algorithm to construct the polynomials in $S_{\pi_{1}}A^{*} \otimes S_{\pi_{2}}B^{*} \otimes S_{\pi_{3}}C^{*}$. While this algorithm is based on classical methods, we refer the reader to the works \cite{LandsbergTensorBook, oeding_thesis, OedingTan} which use similar language for more details. We point out that the complexity of any algorithm to compute polynomials from Schur modules will depend on dimension and degree.  This piece-by-piece algorithm attempts work with the smallest dimensional space possible at each step, thus reducing the complexity and increasing the chances that the computation will finish in a reasonable amount of time.

For concreteness, we fix the degree $d=6$.
The input to the algorithm is the fillings of the tableau of shapes $\pi_{1},\pi_{2},\pi_{3}$.  
The first step is to construct a highest weight vector in $A^{\otimes 6}\otimes B^{\otimes 6} \otimes C^{\otimes 6}$.  For this we work one vector space at a time. Suppose $a_{1},a_{2},a_{3}$ is a basis of $A^{*}$. Then $a_{1}\otimes a_{1}\otimes a_{2}\otimes a_{2}\otimes a_{3} \otimes a_{3}$ is a \emph{pre-highest weight vector}.
The Young symmetrizer 
\[Y_{\pi_{1}}:A^{*}\otimes A^{*}\otimes A^{*}\otimes A^{*}\otimes A^{*}\otimes A^{*}  \rightarrow A^{*}\otimes A^{*}\otimes A^{*}\otimes A^{*}\otimes A^{*}\otimes A^{*}\]
 is the map that skew symmetrizes the vector spaces $A^{*}$ in positions corresponding to the columns of the filling associated to $\pi_{1}$ and then symmetrizes the vector spaces corresponding to the rows of the filling associated to $\pi_{1}$.  The image of the pre-highest weight vector is a highest weight vector of $S_{\pi_{1}}A$ in $(A^{*})^{\otimes 6}$. We perform the analogous construction in the $B^{*}$ and $C^{*}$ factors and take the tensor product of the resulting highest weight vectors.

The resulting vector we have constructed is in 
$S_{\pi_{1}}A^{*}\otimes S_{\pi_{2}}B^{*}\otimes S_{\pi_{3}}C^{*}$,
however it is embedded in $(A^{*})^{\otimes 6} \otimes (B^{*})^{\otimes 6} \otimes (C^{*})^{\otimes 6}$.  
The final step is to perform the re-ordering isomorphism 
$(A^{*})^{\otimes 6} \otimes (B^{*})^{\otimes 6} \otimes (C^{*})^{\otimes 6} \rightarrow (A^{*} \otimes B^{*} \otimes C^{*})^{\otimes 6}$, and then symmetrize the result to arrive at a polynomial in $S^{6}(A^{*}\otimes B^{*}\otimes C^{*})$.

We computed the 10 polynomials in $S_{2,2,2}A^{*} \otimes S_{2,2,2} B^{*} \otimes S_{3,1,1,1} C^{*}$ using the fixed fillings 
\[\young(12,34,56)\;\;,\;\;\young(14,25,36)\] for $\pi_{1}$ and $\pi_{2}$ respectively with each of the following fillings for $\pi_{3}$
\begin{gather*}
\young(111,2,3,4), \young(112,2,3,4), \young(113,2,3,4), \young(114,2,3,4),
\young(122,2,3,4), \young(123,2,3,4), \young(124,2,3,4),
\\
\young(133,2,3,4), \young(134,2,3,4), 
\young(144,2,3,4).
\end{gather*}
Notice that up to re-naming the numbers, the fillings for $\pi_{3}$ can be divided into two classes, depending on whether the last two numbers in the first row are equal.  The four fillings of the first class (with the last two numbers in the first row equal) correspond to polynomials with $936$ terms, whereas the six fillings of the second class correspond to polynomials with $576$ terms. 

Denote by $p_{i,j,k}$, ( $1\leq i,j \leq 3$ and $1\leq k\leq 4$) a basis of  $A^{*}\otimes B^{*}\otimes C^{*} \cong \CC^{3}\otimes \CC^{3}\otimes \CC^{4}$.
Then define the swap $p_{i,j,k} \leftrightarrow p_{i,j,l}$ for fixed $k,l$ and for all $1\leq i,j \leq 3$.  Up to sign, this swap takes the polynomial associated to the filling $\young(1kk,2,3,4)$ to the polynomial associated to the filling 
$\young(1ll,2,3,4)$, and if $m$ is different than $k$ and $l$, the swap takes the polynomial associated to the filling $\young(1km,2,3,4)$ to the one associated to $\young(1lm,2,3,4)$.
This additional symmetry could be useful for the Bertini computation, however our computation completed without the need to implement this symmetry, so we did not use it.  We hope to exploit this for future work.

These fillings produce homogeneous polynomials that are, moreover, homogeneous in multi-degree.  In general, the \emph{multi-degree} of a monomial is a collection of vectors
$[[l^{A}_{1},l^{A}_{2},l^{A}_{3}]$, $[l^{B}_{1},l^{B}_{2},l^{B}_{3}]$, $[l^{C}_{1},l^{C}_{2},l^{C}_{3},l^{C}_{4}]]$, and is defined on a single variable $x_{i,j,k}$ by  $l^{A}_{i'}$ is $0$ (respectively $1$) for $x_{i,j,k}$ if $i\neq i'$ (respectively $i=i'$) ($l^{B}_{j'}$ and $l^{C}_{k'}$ are defined similarly) and the multi-degree is defined for monomials by declaring it to be additive over products of variables.
For example, the following is a sampling of terms in the highest weight polynomial corresponding to the filling $\young(111,2,3,4)$,
\[\dots
-x_{321}x_{113}x_{211}x_{221}x_{134}x_{332} -x_{321}x_{122}x_{231}^2x_{313}x_{114} +x_{211}x_{312}x_{131}x_{121}x_{334}x_{223} 
\dots\]
and one finds that this polynomial has multi-degree $[[2,2,2],[2,2,2],[3,1,1,1]]$.

\begin{remark}\label{M6 basis}
Note that when $a=b=3$ and $c=4$, $S_{2,2,2}A^{*} \otimes S_{2,2,2} B^{*} \otimes S_{3,1,1,1} C^{*}$ is $10$-dimensional.  When $a=b=c=4$, the dimension of $S_{2,2,2}A^{*} \otimes S_{2,2,2} B^{*} \otimes S_{3,1,1,1} C^{*}$ increases to $1000$, however the basis of this larger space can still be constructed from the two polynomials that have 576 and 936 monomials via the type of swap of variables described above for the index $k$ in $p_{ijk}$, but also allowing similar swaps for each of the indices $i$ and $j$.
\end{remark}

\section{Geometric techniques for secant varieties}\label{sec:geometry}
Suppose $A'\subset A$, $B'\subset B$ and $C'\subset C$. Landsberg and Manivel have shown  how to take equations on $\sigma_{r}(\PP A' \times \PP B' \times \PP C')$  to equations on $\sigma_{r}(\PP A \times \PP B \times \PP C)$ and call this procedure \emph{inheritance}  \cite[Proposition~4.4]{LM04}.

Subspace varieties contain tensors that can be written using fewer variables. More specifically,
\begin{align*} \Sub_{a',b',c'} (A\otimes B\otimes C) := \big\{ [T]\in \PP ( A\otimes B\otimes C) \mid  \exists \CC^{a'}\subseteq A, \\
 \CC^{b'} \subseteq B, \CC^{c'} \subseteq C, \text{ with } [T] \in \PP( \CC^{a'}\otimes \CC^{b'}\otimes \CC^{c'} )
\big\}
.\end{align*}
Landsberg and Weyman have shown that 
$\Sub_{a',b',c'} (A\otimes B \otimes C)$ is normal with rational singularities, and the ideal is generated by minors of flattenings \cite[Theorem~3.1]{LWSecSeg}. Recall that a flattening of a $3$-tensor in $A\otimes B\otimes C$ is the choice to view it as a matrix in $A\otimes (B\otimes C)$, $B\otimes (A\otimes C)$ or $(A\otimes B)\otimes C$.

The subspace varieties are important in light of equations because of the fact that
\[\Sub_{r,r,r}(A\otimes B \otimes C) \supseteq \sigma_{r}(\PP A\times \PP B \times \PP C),\] and therefore when non-trivial, the ideal of $\Sub_{r,r,r}$ gives equations of $\sigma_{r}$.  There is an easy test for a module to be in the ideal of a subspace variety, namely $S_{\pi_{1}}A^{*} \otimes S_{\pi_{2}} B^{*} \otimes S_{\pi_{3}}C^{*}$ is in the ideal of $\Sub_{a',b',c'} (A\otimes B \otimes C)$ if and only if at least one of the following holds; $\#(\pi_{1})>a'$, $\#(\pi_{2})>b'$ or $\#(\pi_{3})>c'$, where $\#(\cdot)$ is the number of parts of the partition. 

Landsberg and Manivel made an important reduction for the salmon problem, which we record here.  Friedland pointed out that their proof contained an error, which he corrected in \cite{Friedland2010}.
Let $a,b,c$ respectively denote the dimensions of $A,B,C$.
\begin{theorem}[Landsberg-Manivel, Friedland]\label{LMF}  As sets, for $a,b,c \geq 3$, \\ $\sigma_{4}\left( \PP^{ a-1} \times \PP^{b-1} \times \PP^{c-1}\right)$ is the zero-set of the union of:
\begin{enumerate}
\item Strassen's commutation conditions,
\begin{align*}M_{5} :=
        S_{(3,1,1)}A^{*} \otimes S_{(2,1,1,1)} B^{*} \otimes S_{(2,1,1,1)} C^{*}  \\
 \oplus S_{(2,1,1,1)}A^{*} \otimes S_{(3,1,1)} B^{*} \otimes S_{(2,1,1,1)} C^{*} \\
 \oplus S_{(2,1,1,1)}A^{*} \otimes S_{(2,1,1,1)} B^{*} \otimes S_{(3,1,1)} C^{*},
 \end{align*}

\item Equations inherited from $\sigma_{4}\left( \PP^{2} \times \PP^{2} \times \PP^{3} \right)$, and 

\item Modules in $S^{5}(A^{*}\otimes B^{*}\otimes C^{*})$ containing a $\textstyle{\bigwedge^{5}}$, \emph{i.e.} equations for $\Sub_{4,4,4}$.
\end{enumerate}
 \end{theorem}

Note that when $a=b=c=4$, the third set of equations is trivial.
The key point is that we will have a complete description of the set-theoretic defining equations of $\sigma_{4}(\PP^{3}\times \PP^{3} \times \PP ^{3})$  as soon as we have the equations of $\sigma_{4}(\PP^{2}\times \PP^{2} \times \PP ^{3})$.

\begin{remark}\label{M5 basis}
The equations in degrees $5$ as well as equations in degree $9$ inherited from $\sigma_{4}\left( \PP^{2} \times \PP^{2} \times \PP^{3} \right)$ were found by Strassen \cite{Strassen83} and were described in terms of certain commutation conditions. Later, Landsberg and Manivel \cite{LM08} reinterpreted these conditions from the geometric and representation theoretic point of view and provided generalizations in this language. In \cite{SturmfelsOpenProblems} one finds a nice description of these equations requiring only basic linear algebra.  Analogous to our description of the equations in degree $6$, here we give the representation theoretic description of the polynomials of degree $5$.

Note also that when  $a=b=c=4$,  $M_{5}$ is a $1728$-dimensional irreducible $G$-module, for $G = \left(GL(4)\times GL(4)\times GL(4)\right)\ltimes \mathfrak{S_{3}}$.  A natural basis of $M_{5}$ can be constructed as in the previous section. For this we need to give the fillings for the triple of Young diagrams corresponding to the partitions $(2,1,1,1), (2,1,1,1), (3,1,1)$.
We note that up to permutation, there is just one equivalence class for the fillings of the diagram for $(2,1,1,1)$ with representative $\young(11,2,3,4)$. There are three equivalence classes for the fillings of the diagram for $(3,1,1)$ with representatives $\young(111,2,3)$, $\young(112,2,3)$ and $\young(112,3,4)$.

Therefore to construct representatives for a basis of $S_{(2,1,1,1)}A^{*} \otimes S_{(2,1,1,1)} B^{*} \otimes S_{(3,1,1)} C^{*}$, we fix the representative filling for $(2,1,1,1)$ in both instances, and we let the filling for $(3,1,1)$ vary over the three representatives.
Thus we construct three polynomials, one for each representative filling of the diagram for $(3,1,1)$ and respectively, these polynomials have $180$, $360$ and  $540$ monomials.  A basis of polynomials for one of the $3$ isomorphic modules in $M_{5}$ is contained in the file ``deg\_5\_salmon.txt'' with our ancillary materials.
After constructing these three polynomials, the rest of the polynomials in the basis of $M_{5}$ can be constructed by the substitutions and swaps of variables as mentioned above (see the discussion above Remark \ref{M6 basis}).
\end{remark}

Another important result for the salmon problem is from Strassen, which has been reinterpreted in representation theoretic language in \cite{LM08}.
\begin{theorem}[{\cite{Strassen83}}]\label{thm:Strassen}
The ideal of the hypersurface $\sigma_{4}(\PP^{2}\times \PP^{2} \times \PP^{2}) \subset \PP^{26} $ is generated in degree $9$ by a nonzero vector in the $1$-dimensional module
\[S_{(3,3,3)} \CC^{3} \otimes S_{(3,3,3)}\CC^{3} \otimes S_{(3,3,3)} \CC^{3} .\]
\end{theorem}

Let $M_{9}$ denote the inherited module $S_{(3,3,3)} \CC^{3} \otimes S_{(3,3,3)}\CC^{3} \otimes S_{(3,3,3)} \CC^{4}$.
Inheritance implies that
$M_{9}\in \I(\sigma_{4}(\PP^{2}\times \PP^{2} \times \PP^{3}))$.

\begin{remark}\label{rmk:ottaviani}Suppose $[T]\in \PP(A\otimes B\otimes C)$, with $\dim(A) = 3$.  Then write $T = a_{1}\otimes T_{1} + a_{2}\otimes T_{2}  + a_{3}\otimes  T_{3} $, where the $T_{i}$ are $b\times c$ matrices in $B\otimes C$ and the $a_{i}$ are a basis of $A$.

Strassen described his equation in degree $9$ as follows. On an open set one may assume that $T_{1}$ is invertible. Then consider the polynomial
\[\det(T_{1})^{2}\det(T_{2}T_{1}^{-1}T_{3} - T_{3} T_{1}^{-1}T_{2}).\]
Strassen showed that this polynomial is irreducible, of degree $9$, and vanishes on $\sigma_{4}(\PP A \times \PP B \times \PP C)$. 

A useful reformulation by Ottaviani of Strassen's equation is the following (see \cite{LanOtt, Ott1}). As before write $T = a_{1}\otimes T_{1} + a_{2}\otimes T_{2}  + a_{3}\otimes  T_{3} $.  Here one  does not require any of the slices $T_{1},T_{2},T_{3}$ to be invertible.
Construct the block matrix
\begin{equation}\label{block}
\psi_{T} = \left(\begin{array}{ccc}
0 & T_{3} & -T_{2} \\
-T_{3} & 0 & T_{1} \\
T_{2} & -T_{1} & 0
\end{array}
\right).\end{equation}
One checks that $\psi_{T}$ is linear in $T$, and that if $[T]\in Seg(\PP A \times \PP B \times \PP C)$ then $Rank(\psi_{T})=2$.  Therefore if $[T]$ is a general point in $\sigma_{k}(\PP A \times \PP B \times \PP C)$ it can be written as the sum of $k$ points on $Seg(\PP A \times \PP B \times \PP C)$ so $Rank(\psi_{T})\leq 2k$ by the sub-additivity of matrix rank.  In particular, in the case $\dim(A) = \dim(B) = \dim(C) = 3$, the $9\times 9$ determinant  $\det(\psi_{T})$ gives a non-trivial equation for $\sigma_{4}(\PP A \times \PP B \times \PP C)$, which is also Strassen's equation. This polynomial has $9,216$ monomials.  Note that $\psi_{T}$ is not a skew-symmetric matrix unless the matrices $T_{i}$ are symmetric, otherwise any odd-sized determinant would vanish identically.
\end{remark}

\begin{remark}\label{M9 basis}
In the case that $a=b=3$ and $c=4$, as a vector space, $M_{9}$ is isomorphic to $S^{3}\CC^{4}$ so $\dim(M_{9}) = 20$. 
When the highest weight vector of a module has a determinantal representation (as in the case of $M_{9}$), it is typically much faster to compute a basis of the module from the highest weight vector using lowering operators.  (Lowering operators are standard in the theory of Lie algebras, but are not the focus of this work.  We refer the interested reader to \cite[Section 3.4]{oeding_pm_paper} for an explicit treatment of this method.)
Using this method, we found that the natural basis of $M_{9}$ consists of polynomials with $9,216$ or $25,488$ or $43,668$ monomials. 
 This basis is a 23Mb text file of polynomials, too large to include with our ancillary files due to the restrictions of the arXiv, but may be obtained from either author.  
  As in Remarks \ref{M6 basis} and \ref{M5 basis}, these polynomials can be associated to representative polynomials, depending on fillings.  In the $A$ and $B$-factors, the diagram for $(3,3,3)$ can only have one semi-standard filling, namely $\young(111,222,333)$.  In the $C$-factor, there are three classes of fillings, namely $\young(111,222,333)$, $\young(111,222,334)$ and $\young(111,223,344)$.  These fillings yield the representative polynomials consisting of $9,216$ or $25,488$ or $43,668$ monomials respectively.  The rest of the polynomials in a basis of $M_{9}$ can be constructed by the substitutions and swaps described in our treatment of $M_{6}$ (see the discussion above Remark \ref{M6 basis}).
\end{remark}
Alternately, a basis of $M_{9}$ can be constructed via Ottaviani's formulation.  They are derived from the condition that the now $9\times 12$ matrix appearing in \eqref{block} have rank $8$ or less. However, the space of $9\times 9$ minors of $\psi_{T}$ is no longer irreducible when $a=b=3$ and $c=4$.  Namely the space of $9\times 9$ minors of the $9\times 12$ matrix $\psi_{T}$ is the following representation
\begin{gather*}
S_{3,3,3}A^{*} \otimes S_{3,3,3}B^{*} \otimes S_{3,3,3}C^{*} \\
\oplus 
S_{4,3,2}A^{*} \otimes S_{3,3,3}B^{*} \otimes S_{3,3,2,1}C^{*} \\
\oplus
S_{5,2,2}A^{*} \otimes S_{3,3,3}B^{*} \otimes S_{3,2,2,2}C^{*} 
.\end{gather*}

There are three equivalence classes of maximal minors of $\psi_{T}$ depending only on the column index $I$ of the maximal minor of $\Delta_{I}(\psi_{T})$.  Let $P = (P_{1},P_{2},P_{3})$ be the partition of the set $\{1,\dots,12\}$ into three sets $P_{1}= \{1,2,3,4\}$, $ P_{2} =\{5,6,7,8\}$, $P_{3}=\{9,10,11,12\}$.
The representation 
$S_{3,3,3}A^{*} \otimes S_{3,3,3}B^{*} \otimes S_{3,3,3}C^{*} $ is associated to the minors $\Delta_{I}(\psi_{T})$ such that $|I\cap P_{i}| = 3$ for $i=1,2,3$.  This condition precisely forces the minor of $\psi_{T}$ to be constructed with $3\times 3$ submatrices of $T_{1},T_{2}$ and $T_{3}$.
The representation 
$S_{4,3,2}A^{*} \otimes S_{3,3,3}B^{*} \otimes S_{3,3,2,1}C^{*} $ is associated to the minors $\Delta_{I}(\psi_{T})$ such that $|I\cap P_{1}| = 4$, $|I\cap P_{2}| = 3$, $|I\cap P_{3}| = 2$.
The representation 
$S_{5,2,2}A^{*} \otimes S_{3,3,3}B^{*} \otimes S_{3,2,2,2}C^{*} $ is associated to the minors $\Delta_{I}(\psi_{T})$ such that $|I\cap P_{1}| = 4$, $|I\cap P_{2}| = 4$, $|I\cap P_{3}| = 1$.

Note that the symmetry implied by the fact that $A$ and $B$ have the same dimension allows us to reverse the roles of $A$ and $B$ to find two more modules in the ideal, namely the two modules
$ S_{3,3,3}A^{*} \otimes  S_{4,3,2}B^{*} \otimes S_{3,3,2,1}C^{*} $ and 
$ S_{3,3,3}A^{*} \otimes  S_{5,2,2}B^{*} \otimes S_{3,2,2,2}C^{*} $ must also vanish on $\sigma_{4}(\PP A \times \PP B \times \PP C)$.

While we have described five modules of degree $9$ equations which vanish on $\sigma_{4}(\PP A \times \PP B \times \PP C)$, 
we only use the module
$ M_{9} = S_{3,3,3}A^{*} \otimes S_{3,3,3}B^{*} \otimes S_{3,3,3}C^{*} $ along with $M_{6}$ described above for our set-theoretic defining equations. 
We can conclude that $\langle M_{9}\rangle \not \subset \langle M_{6} \rangle$ by analyzing the shapes of the partitions involved.  More specifically, in the $C$-factor the partition $(3,3,3)$ only has $3$ parts, but if $S_{\pi_{1}}A^{*}\otimes S_{\pi_{2}}B^{*}\otimes S_{\pi_{3}}C^{*}$ is a module in the ideal generated by $M_{6}$ then $\pi_{3}$ must have at least $4$ parts. However this argument fails for the other four degree $9$ modules so it is possible that these equations are in the ideal generated by $M_{6}$. Moreover our set-theoretic result implies that it must be the case that the other degree $9$ modules are in the ideal generated by $M_{6}$ (up to high numerical accuracy).
\begin{example}[\cite{Friedland2010}]\label{ex:deg 16}
Friedland has shown that the known equations in degree $9$ are not sufficient to define $\sigma_{4}(\PP A \times \PP B \times \PP C)$ set-theoretically when $\dim(A)\geq 3$, $\dim(B)\geq 3$ and $\dim(C)\geq 4$. We thank J.M. Landsberg for the following clarification of Friedland's example.
Consider the point
\[
P=(a_{1}\otimes b_{1}+a_{2}\otimes b_{2})\otimes c_{1}+(a_{1}\otimes b_{1}+a_{2}\otimes b_{3})\otimes c_{2}+(a_{1}\otimes b_{1}+a_{3}\otimes b_{2})\otimes c_{3}+(a_{1}\otimes b_{1}+a_{3}\otimes b_{3})\otimes c_{4}
.\]
The span of $\{a_{1},a_{2},a_{3}\}\subset A$ and the span of $\{b_{1},b_{2},b_{3}\}\subset B$ are both no more than 3-dimensional, so $P$ is a zero of $M_{5}$ since the representations $S_{\pi_{1}}A^{*}\otimes S_{\pi_{2}}B^{*}\otimes S_{\pi_{3}}C^{*}$ in $M_{5}$ each have either $|\pi_{1}|=4$ or $|\pi_{2}|=4$, and therefore the respective Schur functor $S_{\pi_{i}}$ with $|\pi_{i}|=4$ will annihilate a $3$-dimensional subspace.
One finds that $\psi_{T}(P)$ has rank $8$ and therefore $P$ is a zero of $M_{9}$. However $P$ is not a point of $\sigma_{4}(\PP A \times \PP B \times \PP C)$.  This geometric argument implies that more polynomials are needed than just the degree $5$ and $9$ equations. For this, Friedland produces equations of degree $16$ which do not vanish on $P$. 
On the other hand, $P$ is not in the zero set of $M_{6}$, so $M_{6}$ is sufficient to rule out the possibility of points of the same form as $P$ to have border rank $4$.
Therefore, one could repeat Friedland's proof, modifying the argument where he uses degree 16 equations with these degree $6$ equations and thus obtain a new result, and a computer-free proof of Theorem \ref{thm:main}.
\end{example}

\begin{remark} 
To construct a basis of the $8000$ dimensional space $S_{(3,3,3)} \CC^{4} \otimes S_{(3,3,3)}\CC^{4} \otimes S_{(3,3,3)} \CC^{4}$, one can repeat the lowering operator procedure.
Since these polynomials are very complicated, our experience is that, in practice, one should use the degree $9$ equations in their determinantal form.  In particular, to check if a point $z$ vanishes on all of the polynomials in $S_{(3,3,3)} \CC^{4} \otimes S_{(3,3,3)}\CC^{4} \otimes S_{(3,3,3)} \CC^{4}$, it is more efficient to first construct the matrix in \eqref{block} for the point $z$ and check that the determinant vanishes.  Then repeat this test for all allowable changes of coordinates, in other words, for every  $g \in GL(4)\times GL(4) \times GL(4)$ construct the matrix in \eqref{block} for $g.z$ and check that the determinant still vanishes.  (This is sufficient because our module is the span of the orbit of a single polynomial.) Moreover, if one only wants a quick check that $z$ is in the zero-set with high probability, it suffices to check that $g.z$ is in the zero-set for a random $g$.  In this quick test, a non-vanishing result is certain, but vanishing must be re-verified with an exact (non-randomized) test.
\end{remark}

Since $(2,2,2)$ has $3$ parts, and $(3,1,1,1)$ has $4$ parts, $M_{6}$ must vanish on the subspace varieties $\Sub_{2,3,4} \cup  \Sub_{3,2,4} \cup  \Sub_{3,3,3}$.
Also, note that two of these subspace varieties are already contained in the secant variety, namely $\sigma_{4}\left( \PP^{2} \times \PP^{2} \times \PP^{3} \right) \supset \Sub_{2,3,4} \cup  \Sub_{3,2,4}$. Indeed, if $x \in \Sub_{2,3,4}$, there exists $A'\subset A$ such that $\dim(A')=2$ and  $x \in \PP(A'\otimes B \otimes C)$.  But in this case $\PP(A'\otimes B \otimes C) = \sigma_{4}(\PP A'\times \PP B \times \PP C) \subset \sigma_{4}(\PP A\times \PP B \times \PP C)$.  The same argument is repeated for $\Sub_{3,2,4}$.

If $M$ is a set of polynomials, let $\V(M)$ denote the zero-set of $M$. Based on the above evidence, we make the conjecture
\begin{conj}\label{conj:m6} As sets,
\[\V(S_{(2,2,2)}\CC^{3} \otimes S_{(2,2,2)} \CC^{3} \otimes S_{(3,1,1,1)} \CC^{4})= \sigma_{4}\left( \PP^{2} \times \PP^{2} \times \PP^{3} \right) \cup \Sub_{3,3,3}.\]
\end{conj}

Computation~\ref{comp:Bertini} below verifies that Conjecture \ref{conj:m6} is true up to high numerical accuracy.

\begin{theorem}[Corollary to Computation~\ref{comp:Bertini}]\label{cor:main}
Let $A \cong \CC^{3}$, $B \cong \CC^{3}$, $C \cong \CC^{4}$. Up to high numerical accuracy,
the secant variety $\sigma_{4}\left( \PP A \times \PP B \times \PP C \right)$ is defined set-theoretically by
\begin{eqnarray*}
M_{6} =  S_{(2,2,2)} A^{*} \otimes S_{(2,2,2)} B^{*} \otimes S_{(3,1,1,1)} C^{*}\\
M_{9} = S_{(3,3,3)} A^{*} \otimes S_{(3,3,3)} B^{*} \otimes S_{(3,3,3)} C^{*} 
.\end{eqnarray*}
\end{theorem}
\begin{proof}
By Proposition~\ref{prop:deg6ideal} and by Strassen's Theorem~\ref{thm:Strassen} combined with inheritance we know that both $M_{6}$ and $M_{9}$ are in the ideal of $\sigma_{4}\left( \PP A^{*} \times \PP B^{*} \times \PP C^{*} \right)$. 
So we know that $\sigma_{4}\left( \PP A \times \PP B \times \PP C \right) \subset \V(M_{6}\oplus M_{9})$.

For the other inclusion, select a point $z$ in the common zero locus of $M_{6}$ and $M_{9}$. Since $z\in \V(M_{6})$, Conjecture \ref{conj:m6} says that either $z$ is on the secant variety, in which case we are done, or $z$ is on the subspace variety.  In the latter case, let $C' \subset C$ be a $3$-dimensional vector space so that $z\in \PP (A\otimes B\otimes C')$. Then $z$ is a zero of $M_{9} = S_{(3,3,3)} A^{*} \otimes S_{(3,3,3)} B^{*} \otimes S_{(3,3,3)} C^{*}$, and therefore is also a zero of the polynomials in the restriction $S_{(3,3,3)} A^{*} \otimes S_{(3,3,3)}B^{*} \otimes S_{(3,3,3)} C'^{*}$.  So by Strassen's Theorem \ref{thm:Strassen}, 
\[z \in \sigma_{4}(\PP A\times \PP B \times \PP C') \cong  \sigma_{4}(\PP^{2}\times \PP^{2} \times \PP ^{2})  .\]
We conclude because we have the obvious inclusion 
\[
\sigma_{4}(\PP A\times \PP B \times \PP C') \subset \sigma_{4}(\PP A\times \PP B \times \PP C)
.\]
\end{proof}
We used numerical algebraic geometry, specifically Bertini, to compute the decomposition of the zero set $\V(M_{6})$ into irreducible varieties.  We outline this computation in the next section. However, if one were to prove Conjecture \ref{conj:m6}, then the qualifier ``with high numerical accuracy'' may be removed from the statement of Theorem \ref{cor:main}.

Recall that the Landsberg-Manivel-Friedland Theorem \ref{LMF} above said that set-theoretic defining equations of $\sigma_{4}\left( \PP^{ a-1} \times \PP^{b-1} \times \PP^{c-1}\right)$ with  $a,b,c \geq 3$ will be known as soon as set-theoretic defining equations of $\sigma_{4}\left( \PP^2 \times \PP^2 \times \PP^3\right)$ are known,  and this is the content of Theorem~\ref{cor:main}.  Therefore we can restate the immediate consequence of combining Theorem \ref{LMF} with our computations:
\begin{theorem}\label{thm:main} As sets, for $a,b,c \geq 3$, up to high numerical accuracy,  $\sigma_{4}\left( \PP^{ a-1} \times \PP^{b-1} \times \PP^{c-1}\right)$,  is the zero-set of:
\begin{enumerate}
\item  Strassen's commutation conditions,
\begin{align*}M_{5} :=
        S_{(3,1,1)}A^{*} \otimes S_{(2,1,1,1)} B^{*} \otimes S_{(2,1,1,1)} C^{*}  \\
 \oplus S_{(2,1,1,1)}A^{*} \otimes S_{(3,1,1)} B^{*} \otimes S_{(2,1,1,1)} C^{*} \\
 \oplus S_{(2,1,1,1)}A^{*} \otimes S_{(2,1,1,1)} B^{*} \otimes S_{(3,1,1)} C^{*},
 \end{align*}
\item equations inherited from $\sigma_{4}\left( \PP^{2} \times \PP^{2} \times \PP^{3} \right)$,
\begin{align*}
M_{6} =  S_{(2,2,2)} A^{*} \otimes S_{(2,2,2)} B^{*} \otimes S_{(3,1,1,1)} C^{*}\\
M_{9} = S_{(3,3,3)} A^{*} \otimes S_{(3,3,3)} B^{*} \otimes S_{(3,3,3)} C^{*}, 
\end{align*}

\item and modules in $S^{5}(A^{*}\otimes B^{*}\otimes C^{*})$ containing a $\textstyle{\bigwedge^{5}}$, \emph{i.e.} equations for $\Sub_{4,4,4}$.
\end{enumerate}
\end{theorem}
\begin{remark}
The qualifier ``up to high numerical accuracy'' can be removed if one uses Friedland's argument \cite{Friedland2010} modified by our computations, as mentioned in Example \ref{ex:deg 16}.
\end{remark}

\section{Results using numerical algebraic geometry}\label{sec:Bertini}

In this section, we provide a brief overview of the basic methods of numerical algebraic geometry; references for further details are provided.  We then describe the results of the run establishing the main result of this article and conclude with a short discussion regarding the reliability of numerical algebraic geometry methods and, more to the point, the reliability of this result.

\subsection{Brief overview of numerical algebraic geometry methods}

Given generators of an ideal of $\mathbb C[x_1,\ldots,x_N]$, the methods of numerical algebraic geometry will produce 
a {\em numerical irreducible decomposition} for the associated variety $X\subset\mathbb C^N$.  In particular, for each irreducible component $Z$ of $X$, 
these methods will produce $\deg Z$ numerical approximations (to any number of digits) of generic points on $Z$.  The end result is 
a catalog of all irreducible components of $X$, each indicated by a set of {\em witness points} on the component (together referred to as a {\em witness 
set} for the component), its dimension, and its degree.

The core method of numerical algebraic geometry is {\em homotopy continuation}, a method for approximating the complex zero-dimensional solution set of a 
polynomial system.  The basic idea of homotopy continuation is to cast the given polynomial system $F$ 
as a member of a parameterized family of polynomial systems, one of which, $G$, has known solutions or is otherwise easily solved.  If done correctly, the solutions of $G$ will 
vary continuously to those of $F$ as the parameters are varied appropriately.  By tracking these paths numerically (using predictor-corrector methods), one will 
arrive at numerical approximations of all complex zero-dimensional solutions of $F$.  There have been many technical advances in this area that contribute heavily 
to the reliability of these methods.  See~\cite{SW05,Li} for general references and \cite{AMP,AMP2} regarding the use of adaptive precision methods for added reliability.

Pairing homotopy continuation with the use of hyperplane sections, monodromy, and a few other methods described fully in~\cite{SW05} yields the numerical irreducible decomposition.  Briefly, 
a $d$-dimensional irreducible algebraic set in $\mathbb C^N$ will intersect a generic codimension $d$ linear space in a set of points.  This statement about genericity 
(along with similar assumptions of genericity throughout numerical algebraic geometry) is the reason for referring to these methods as {\em  probability-one} methods, as 
described further below.  

The computation of a numerical irreducible decomposition begins by searching for codimension one irreducible components (by adding $N-1$ linear polynomials to the set of generators and solving for zero-dimensional components via homotopy continuation), followed by codimension two components, etc.  Once this sweep through all possible dimensions has been completed, 
we have a superset of the desired numerical irreducible decomposition, since a linear variety of codimension $d$ will intersect any component of dimension $d$ {\em 
or higher}.  Sommese, Verschelde, and Wampler (and others) have developed methods for removing points in the ``wrong dimension,'' i.e., those discovered while searching for 
components in dimension $d$ which actually lie on higher-dimensional components, called {\em junk points}.  They have also developed algorithms for performing pure-dimensional decompositions to yield witness sets on each irreducible component (instead of the initially-found witness sets for the union of all equidimensional irreducible components).  See~\cite{SW05} for further details.

There are three main software packages in this field: Bertini~\cite{Bertini}, HOM4PS-2.0~\cite{HOM4PS}, and PHCpack~\cite{PHC}.  Each package has various benefits over the 
others~\cite{BertComp}.  Since Bertini is typically the most efficient package for large, parallel, positive-dimensional problems as well as the package with the most reliability and precision features, we used Bertini in our computations for this article.  In fairness, it should also be noted that Bates is a Bertini developer.

\subsection{Numerical results for the Salmon Problem}

\begin{computation}\label{comp:Bertini}
Up to $10$ digits of accuracy, 
the zero-set of the $10$ polynomials in a basis of $M_{6}$ (defined above) has precisely two irreducible components.  One, in dimension 31, has degree 345.  The other, in dimension 29, has degree 84.
\end{computation}

Indeed,
$ \sigma_{4}\left( \PP^{2} \times \PP^{2} \times \PP^{3} \right) $ is non-defective and has dimension $31$ \cite[Theorem 4.6]{AOP}.
It is also straightforward to check that $ \Sub_{3,3,3}$ has dimension $29$, and by the pigeon-hole principle, these must be our components in the zero-set of $M_{6}$. Though these dimensions are sufficient information to identify our varieties, as additional information, we find that this secant variety has degree $345$ and the subspace variety has degree $84$.

\begin{proof}
The conclusion comes from the results of a calculation on Bertini \cite{Bertini} using approximately 2 weeks of computing time on 8 processors, using tight controls including small tracking and final tolerances ($10^{-10}$ or smaller), adaptive precision numerical methods, and a variety of checks and error controls built into Bertini (such as checking at $t=0.1$ that no paths have crossed). The output of our computation is included in the files ``main\_data.txt'' and ``screen\_out.txt'', which may be obtained with the other ancillary materials as mentioned above.
\end{proof}
\begin{remark}
After the submission of the first draft of this paper, Friedland provided a counterexample to \cite[Proposition~5.4]{LM08}, thus invalidating the proof of \cite[Corollary~5.6]{LM08}.  Because we quote this result in the present paper, we tried to use numerical methods to find out what could be true. In particular, we tried to find a corrected statement for \cite[Proposition~5.4]{LM08} by computing the zero-set of $S_{(3,1,1)}A^{*} \otimes S_{(2,1,1,1)} B^{*} \otimes S_{(2,1,1,1)} C^{*}$  using Bertini. This proved to be a very expensive computation.  This module, in its smallest form, has a basis of $96$ degree $5$ polynomials in $48$ variables (see the file ``deg\_5\_salmon.txt'' in our ancillary files).  After one month of computational time on 72 processors, we had only completed the first 10 codimensions, tracking up to 2 million paths for each codimension.  An easy geometric argument implies that the smallest possible component has projective dimension 8, indicating that we were very far from completing the computation. In the mean time, a second version of \cite{Friedland2010} appeared, with a corrected proof for \cite[Corollary~5.6]{LM08}.  Due to limited computational resources and time, we decided to abandon further computation and accept the computer-free proof of \cite[Corollary~5.6]{LM08} in \cite{Friedland2010}.
\end{remark}

\subsection{Reliability of this result}

Can Computation \ref{comp:Bertini} be accepted as absolute proof?  No, unfortunately, it may not.  However, this numerical computation gives extremely strong evidence that Computation~\ref{comp:Bertini} is indeed true even without the phrase ``with high numerical accuracy''.  

There are two types of approximations that are used in order to compute the numerical decomposition of a zero-set.  One is the choice of a set of random hyperplanes which cut the space and allow one to look for zero-dimensional solutions to a set of equations. The other type of approximation is the numerical homotopy continuation method which actually searches for the zero-dimensional solutions.

The choice of random hyperplanes amounts to the choice of random numbers from a Zariski open, dense set $S$ of some parameter space rather than choosing some set of points in the complement of $S$.  Since the complement of $S$ is an algebraic set, we know that it must have positive codimension, making it a set of measure zero for any reasonable choice of measure.  Thus, the set of hyperplanes that fail in that they would cause us to miss a component in the zero-set has measure zero, and we say that the choice of hyperplanes will yield the correct result with probability one. 

The second type of approximation that is done in this type of computation is the heart of Bertini and is thoroughly described in \cite{SW05}.  Bertini allows one to set desired accuracy to arbitrary levels, and any computational errors (such as path crossing) are reported. Further, Bertini has additional features such as adaptive precision path tracking, which increase security, \cite{AMP,AMP2}. 

The run for this article used a special equation-by-equation algorithm called regeneration~\cite{regen}.  The run required the following of more than $200,000$ paths and there were no path failures and no crossed paths detected.
In addition, there were no errors in the monodromy or trace test procedures.  The numerical output of our run is contained in the files ``main\_data.txt'' and ``screen\_out.txt'' with our ancillary materials.

We cannot conclude with unquestionable certainty that Computation~\ref{comp:Bertini} holds unconditionally, but we {\em can} state with an extremely high level of confidence that it is correct.  Motivated by this result, we hope to find a direct argument to prove Conjecture~\ref{conj:m6}. 

\subsection{Numerical vs. symbolic computation}

Finally, one might wonder why we chose to use numerical methods to test this conjecture rather than symbolic methods that will provide certainty.  The main reasons are simple: time and space.  Regarding time we expect that without additional ideas to reduce the difficulty of computation, a related calculation using symbolic methods should take at least eight times as long as the calculation in Bertini because Gr\"obner basis algorithms are not completely parallelizable (but for an example of recent progress on this front see \cite{Kredel}). In fact, based on the timings from an ongoing benchmarking project between the Bertini and Singular~\cite{Sing} development teams, we suspect that any symbolic computation will actually take far more than eight times as long.  Regarding the issue of space we must consider data storage at intermediate stages.  While the initial input and final result may be relatively small,  Gr\"obner basis algorithms typically must store large intermediate results for subsequent calculations. On the other hand, homotopy continuation algorithms require a trivial amount of extra data in intermediate stages.  Indeed, the amount of memory used grows linearly with the number of paths tracked (simply because the final point on each path must be stored).   Bertini is thus much less likely to fail due to memory constraints. 

Finally, one could also hope for a (symbolic) certificate of the validity of results obtained by numerical methods. At the EACA School in Tenerife, Spain, Wolfram Decker told us that the development of such certificates is among the current goals of the Singular team, and we hope to be able to use this feature in future work.

\bibliographystyle{amsalpha}
\bibliography{master_bibdata}
\end{document}